\documentclass[12pt]{amsart}

\usepackage{amsfonts,
latexsym,
amssymb,
}

\newcommand{\labbel}{\label}

\newtheorem{theorem}{Theorem}

\newtheorem{prop}[theorem]{Proposition} 
\newtheorem{proposition}[theorem]{Proposition} 
 
\newtheorem{corollary}[theorem]{Corollary}

\theoremstyle{definition}
\newtheorem{definition}[theorem]{Definition}

\newtheorem{problems}[theorem]{Problems} 

\theoremstyle{remark}
\newtheorem{remark}[theorem]{Remark}

\newcommand{\brfrt}{\hspace{0 pt}}

\begin{document}
 
\title
{Products of sequentially pseudocompact spaces}

\author{Paolo Lipparini} 
\address{Dipartimento di Matematica\\Viale della Ricerca  Scient\` \i fica\\II Universit\`a Gelmina di Roma (Tor Vergata)\\I-00133 ROME ITALY}
\urladdr{http://www.mat.uniroma2.it/\textasciitilde lipparin}

\thanks{We wish to express our gratitude to X. Caicedo
and S. Garc\'ia-Ferreira for stimulating discussions and correspondence.
We thank our students from Tor Vergata University for (sometimes) stimulating questions}  



\begin{abstract}
We show that the product
of any number of sequentially pseudocompact 
topological spaces is still sequentially pseudocompact.
The definition of sequential pseudocompactness
can be given in (at least) two ways: we show their equivalence.

Some of the results of the present note already appeared in \cite{stseq}.
\end{abstract}

\maketitle

According to \cite[Definition 1.8]{AMPRT}, a topological space $X$ is 
\emph{sequentially pseudocompact} if and only if, for any sequence 
$( O_n) _{n \in \omega } $ 
 of pairwise disjoint nonempty
open sets of $X$, there are an infinite set $J \subseteq \omega $
and a point $x \in X$ such that every
neighborhood of $x$ intersects all but finitely many elements of 
$( O_n) _{n \in J} $.

In \cite{AMPRT} $X$ is assumed to be a Tychonoff space; 
however, here we assume no separation axiom, if not otherwise stated.
Thus, in the present note, a topological space is \emph{sequentially pseudocompact}
if and only if it satisfies the above condition, regardless of the separation
axioms it satisifes.

In this note we present some results about sequential pseudocompactness.
Some of them are known, at least for one of the possible definitions
of sequential pseudocompactness: see \cite{stseq}.

In fact, in \cite{CKS, stseq} a slightly different definition of sequential pseudocompactness
is given, not requesting  the sets $( O_n) _{n \in \omega } $ to be pairwise disjoint.
This notion is called \emph{sequential feeble compactness} in \cite{stseq}.
In the next proposition we show  that the two definitions  are actually equivalent.
 This turns out to be somewhat useful.

\begin{proposition} \labbel{equiv} 
For every topological space $X$, the following conditions are equivalent.
  \begin{enumerate}  
  \item $X$ is sequentially pseudocompact.
\item For any sequence 
$( O_n) _{n \in \omega } $ 
 of  nonempty
open sets of $X$, there are an infinite set $J \subseteq \omega $
and a point $x \in X$ such that, for every
neighborhood $U$ of $x$,
the set $\{ n \in J \mid U \cap O_n = \emptyset \}$ is finite.
  \end{enumerate}
\end{proposition}

  \begin{proof}
(2) $\Rightarrow $  (1) is trivial.

For the converse, suppose that $X$ is sequentially pseudocompact,
and let $( O_n) _{n \in \omega } $ 
be a sequence of  nonempty
open sets of $X$. 
Suppose by contradiction that

(*) for every infinite set $J \subseteq \omega $
and every point $x \in X$ there is some
neighborhood $U (J,x) $ of $x$ such that 
$N(J,x)=\{n \in J \mid  U(J,x) \cap O_n = \emptyset \}$ 
is infinite. Without loss of generality, we can assume
that $U (J,x)$ is open.

We shall construct by symultaneous induction a sequence
$(m_i) _{ i \in \omega } $ of distinct natural numbers, 
a 
 sequence 
$(J_i) _{i \in \omega } $ of infinite subsets of $ \omega$, 
and a sequence  
of pairwise disjoint nonempty open sets 
$(U_i) _{i \in \omega } $ such that 
  \begin{enumerate} 
   \item[(a)]
$ U _{i} \subseteq O _{m_i} $ for every $i \in \omega $, 
\item[(b)]
$ U _{i} \cap O _{n} = \emptyset  $, for every $i \in \omega $, and  $n \in J_i$, and 
\item[(c)]
$J_i \supseteq J_h$, whenever $i \leq h \in \omega $.  
 \end{enumerate}

Put $m_0=0$ and
pick $x_0 \in O_0$ (this is possible, since 
$O_0$ in nonempty).
Apply (*) with $J= \omega $ and $x=x_0$,
and let $U_0= U( \omega, x_0) \cap O_0 \subseteq O_0 = O _{m_0} $
and $J_0 =N( \omega ,x_0)$.
$U_0 $ is nonempty, since $x_0 \in U( \omega, x_0)\cap O_0$.
By (*), $J_0$ is infinite, and 
Clause (b) is satisfied for $i=0$. 
Hence the basis of the induction is completed.

Suppose now that $ 0 \not=i \in \omega$, and that  we have constructed 
finite sequences
$(m_k) _{ k < i} $, 
$(J_k) _{k< i} $, and
$(U_k) _{k < i } $ satisfying the desired properties.
Let $m_i$ be any element of $ J _{i-1} $.
Since $ J _{i-1} $ is infinite, 
we can choose $m_i$ distinct from all the  
$m_k$'s, for $k < i$ (however, this would follow automatically 
from (a) and (b)).
Let $x_i$ be any element of the nonempty  $O _{m_i}$.  
Apply (*) with $J= J _{i-1}$ and $x=x_i$,
and let $U_i= U (J _{i-1},x_i) \cap O _{m_i}  \subseteq  O _{m_i} $
and $J_i =N(J _{i-1},x_i)$.
As above,
$U_i$ is nonempty, since $x_i \in U (J _{i-1},x_i) \cap  O _{m_i}$.
By the definition of 
$N(J _{i-1},x_i)$, we have that
$J_i \subseteq J _{i-1} $, hence Clause (c) holds, by
the inductive hypothesis.  
By (*), $J_i$ is infinite, and moreover
Clause (b) is satisfied for $i$. 
It remains to show that
$U_i$ is disjoint from $U_k$, for $k < i$.   
Since, by construction, $m_i \in  J _{i-1}  $,
then, by (c) of the inductive hypothesis,
for every $k < i$, we have that $m_i \in J_k$, hence, by (b),
$U_k \cap O _{m_i} = \emptyset   $,
hence also
$U_k \cap U_i = \emptyset   $,
since by construction
$U_i \subseteq  O _{m_i} $. 
The induction step is thus complete.

Having constructed sequences satisfying the above properties,
 we can apply sequential pseudocompactness to the sequence 
$(U_i) _{i \in \omega } $ of nonempty pairwise disjoint
open sets,
getting some $J \subseteq \omega $
and some $x \in X$ such that every
neighborhood of $x$ intersects all but finitely many elements of 
$( U_i) _{i \in J} $.

If we put $J^* = \{ m_i \mid i \in J \} $,
then   
every
neighborhood of $x$ intersects all but finitely many elements of 
$( O_n) _{n \in J^*} $, because of Clause (a). We have reached a contradiction, 
thus the proposition is proved.
 \end{proof}

We now state the main result of the present note.
In case sequential pseudocompactness is defined according to
Condition (2) in Proposition \ref{equiv}, it is Theorem 4.1 in \cite{stseq}. 
If not otherwise specified, a product of topological spaces is 
always endowed with the Tychonoff product topology
(the coarsest topology which makes the projections continuous).

\begin{theorem} \labbel{infprod} 
A product of topological spaces is sequentially pseudocompact  if and only if each factor is sequentially pseudocompact.
\end{theorem}  

\begin{proof}
The only-if part is trivial, since any continuous image of 
a sequentially pseudocompact space is sequentially pseudocompact.

In order to prove the converse, we shall use the equivalent formulation
given by Condition (2) in Proposition \ref{equiv}. 
We first show:

\smallskip

(A) \emph{A product of topological spaces is sequentially pseudocompact  if and only if every subproduct of $ \leq \omega $  factors is sequentially pseudocompact.} 

\smallskip
 
The proof 
is similar to the classical argument showing that a topological space
is pseudocompact if and only if every subproduct of $ \leq \omega $  factors is  pseudocompact.

Again, one implication is trivial. To prove the converse, assume 
that $X = \prod _{h \in H} X_h$,
and that  $\prod _{h \in C} X_h$
is sequentially pseudocompact, for every 
$C \subseteq H$ such that $|C| \leq \omega $. 

Let 
$( O_n) _{n \in \omega } $ 
be a sequence of  nonempty
open sets of $X$. 
For every $ n \in \omega$ 
there is a nonempty open set 
$O'_n \subseteq O_n$ 
such that $O'_n$ has the form 
$\prod _{h \in H} O' _{n,h} $, where each
$O' _{n,h} $ is a (nonempty) open set of $X_h$,
and $O' _{n,h} = X_h$, for all but finitely many $h$'s. 
The set 
$K=\{ h \in H \mid O' _{n,h} \not= X_h, \text{ for some } n \in \omega  \}$
 is countable, being a countable union of finite sets.
For every $ n \in \omega$,
the set  $\prod _{h \in K} O' _{n,h} $
is an open set of 
$\prod _{h \in K} X_h$.
By assumption, 
$\prod _{h \in K} X_h$ is sequentially pseudocompact, hence,
by Proposition \ref{equiv}(2), 
  there are 
$J \subseteq \omega $ and  
$ (x_h) _{h \in K} \in \prod _{h \in K} X_h$
such that, for every neighborhood $U$ of 
$ (x_h) _{h \in K} $ in $ \prod _{h \in K} X_h$,
the set $\{ n \in J  \mid U \cap \prod _{h \in K} O' _{n,h} = \emptyset \}$ 
is finite.
Extend the sequence $(x_h) _{h \in K}$ to a sequence 
$(x_h) _{h \in H}$ by picking  some arbitrary 
 $x_h$, for each $h \in H \setminus K$.  
Then,
for every neighborhood $U'$ of $(x_h) _{h \in H}$, 
the set $\{ n \in J  \mid U' \cap \prod _{h \in H} O' _{n,h} = \emptyset \}$ 
is finite, 
since the image $U$ of $U'$ under the natural projection onto
 $ \prod _{h \in K} X_h$ is a neighborhood of 
$(x_h) _{h \in K}$, and, for every $ n \in \omega$, 
$U' \cap \prod _{h \in H} O' _{n,h} = \emptyset $
if and only if  
$U \cap \prod _{h \in K} O' _{n,h} = \emptyset $,
since $O' _{n,h} = X_h$, for every $ n \in \omega$ 
and every $h \in H \setminus K$.
According to Proposition \ref{equiv}(2), 
$ \prod _{h \in H} X_h$ is  sequentially pseudocompact, 
as witnessed by 
$(x_h) _{h \in H}$ and $J$,
and  (A) is proved.

\smallskip

In view of (A), in order to prove the theorem, it is enough to prove that
a product of $ \leq \omega$ sequentially pseudocompact spaces is
sequentially pseudocompact. This is  similar to the proof that
a countable product of sequentially compact spaces is sequentially compact.

Suppose that $X = \prod _{i \in \omega } X_i$
is a product of sequentially pseudocompact spaces,
and that
$( O_n) _{n \in \omega } $ 
is a sequence of  nonempty
open sets of $X$. 
For every $ n \in \omega$ 
there is a nonempty open set 
$O'_n \subseteq O_n$ 
such that $O'_n$ has the form 
$\prod _{i \in \omega } O' _{n,i} $, where each
$O' _{n,i} $ is a (nonempty) open set of $X_i$.

We shall construct a sequence of elements
$ (x _{i} ) _{i \in \omega } $, where each
$ x _{i} $ belongs to $ X _{i} $,
and a sequence $ (J_{i}) _{i \in \omega }  $ of infinite subsets of $ \omega$ 
such that $ J_{i}  \supseteq  J_{j}$ for $i \leq j < \omega $,
and such that, for every $ i \in \omega$,
and every neighborhood $U_i$ of $x_i$ in   
$ X _{i} $, the set 
$\{ n \in J_i  \mid U_i \cap O' _{n, i} = \emptyset \}$ 
is finite.

Since $ X _{0} $
is sequentially pseudocompact, and the 
$O' _{n, 0}$'s are nonempty, 
by Proposition \ref{equiv}(2) 
there
are an infinite set $J_0 \subseteq \omega $
and some
$ x_0 \in X _{0} $
such that, for every neighborhood $U_0$ of $x_0$, 
the set $\{ n \in J_0  \mid U_0 \cap O' _{n, 0} = \emptyset \}$ 
is finite.

Going on, suppose that we have constructed 
$x_i$ and $J_i$. 
By applying the 
sequential pseudocompactness of  $ X _{i+1} $ 
to the (infinite) sequence 
$(O' _{n, i+1}) _{n \in J_i} $, 
we get
 an infinite set $J _{i+1}  \subseteq J_i $
and 
$ x _{i+1} \in X _{i+1 } $
such that, for every neighborhood $U_{i+1} $ of $ x _{i+1}$, 
 the set $\{ n \in J _{i+1}  \mid U_{i+1} \cap O' _{n, i+1}= \emptyset \}$ 
is finite. The desired sequences are thus constructed.

Pick some $n_0 \in J_0$, and,  for each $ i \in \omega$,
 pick some $n_i \in J_i \setminus \{ n_0, \dots, n _{i-1}  \} $,
and let $J= \{ n_i \mid i \in \omega \} $.
Notice that, for every $i \in \omega$,
$J_i \setminus J$ is finite, since 
 the sequence $(J_i) _{i \in \omega } $ is decreasing
with respect to inclusion. 

Let $U$ be any neighborhood of 
$(x_i) _{ i \in \omega } $ in $\prod _{ i \in \omega } X_i$.
Thus, $U$ contains some product 
$U'=\prod _{i \in \omega } U_i$,
where 
 $F=\{ i \in \omega  \mid U_i \not= X_i\}$ is finite, and 
each $U_i$ is a neighborhood of $x_i$ in $X_i$.
If $i \in \omega $ and $i \not \in F$, then  $U_i = X_i$,
hence $U_i \cap O' _{n, i} \not= \emptyset   $, for every 
$ n \in \omega$, since each 
$O' _{n, i}$ is nonempty.
If $i \in F$, then,
since $U_i$ is a neighborhood of $x_i$,
we get by the above construction that 
the set $\{ n \in J_i  \mid U_i\cap O' _{n, i} = \emptyset \}$ 
is finite.
Since all members of $J$ belong to $J_i$,
except possibly for a finite number of elements, 
then also  
$N_i (U')= \{ n \in J  \mid U_i \cap O' _{n, i} = \emptyset \}$ 
is finite.
Let $N (U')$ be the finite set $  \bigcup _{i \in F} N_i(U') $. 
We have proved that if 
$i \in \omega  $
and  $n \in J \setminus N(U')$, 
then 
$U_i \cap O' _{n, i} \not = \emptyset $,
hence,
if $n \in J \setminus N(U')$, then
$U \cap O_n
\supseteq 
\left(\prod _{i \in \omega } U_i  \right) \cap 
\left(\prod_{i \in \omega } O' _{n, i}\right) 
= 
\prod _{i \in \omega } (U_i  \cap  O' _{n, i}) 
\not = \emptyset$.

In conclusion, we have showed that, for every 
sequence  $( O_n) _{n \in \omega } $ 
 of  nonempty
open sets of $X$,
there are $J \subseteq \omega $ 
and an element 
$ (x_i) _{i \in \omega } \in X=\prod _{i \in \omega } X_i$
such that, for every neighborhood $U$ of 
$ (x_i) _{i \in \omega }$, the set 
$\{ n \in J \mid U \cap O_n = \emptyset \} $ is finite
(being a subset of  $N(U'))$.
This is equivalent to sequential pseudocompactness of
$X$, according to Condition (2)
in Proposition \ref{equiv}. 
\end{proof}

Both $\beta \omega $ and 
$D^{ \mathfrak c}$  
are classical examples of 
compact non sequentially compact spaces.
As noticed on \cite[p. 7]{AMPRT},
$\beta \omega $ is not sequentially pseudocompact.
On the other hand, $D^{ \mathfrak c}$
is sequentially pseudocompact, by Theorem \ref{infprod}. 
Thus, compactness together with sequential pseudocompactness 
do not necessarily imply sequential compactness.
In particular, normality and sequential pseudocompactness 
do not imply sequential compactness (thus the result that normality and pseudocompactness 
imply countable compactness cannot be generalized
in the obvious way).
Also, a compact subspace of a compact
sequentially pseudocompact space is not necessarily
sequentially pseudocompact,
since $\beta \omega $ can be embedded in
$D^{ \mathfrak c}$.
In particular, a closed subspace of a sequentially pseudocompact 
space is not necessarily sequentially pseudocompact.

\begin{remark} \labbel{rmk} 
A remark   on terminology is now needed, since we are assuming no separation
axiom, while usually pseudocompactness is considered
in conjunction with the Tychonoff separation axiom.
 There are many conditions which are equivalent to pseudocompactness
in the class of Tychonoff spaces,
 but which are in general distinct, 
for spaces satisfying weaker separation axioms. See, e.~g., \cite{S}. 
A topological space is \emph{feebly compact} if and only if, for any sequence 
$( O_n) _{n \in \omega } $ 
 of nonempty
open sets of $X$, there is a point $x \in X$ such that 
$\{ n \in  \omega  \mid U \cap O_n \not= \emptyset \} $ is infinite,
for every neighborhood $U$ of $x$.
Feeble compactness is also equivalent to 
a notion called \emph{weak initial $ \omega$-compactness}.
See, e.~g., \cite[Remark 3]{L}, and further references there.
For Tychonoff  spaces, feeble compactness is well-known to be equivalent to pseudocompactness.
\end{remark}

\begin{corollary} \labbel{p} 
Let $\mathcal P$ be any property of topological spaces
such that every feebly compact (resp., Tychonoff pseudocompact) topological space satisfying $\mathcal P$ 
is sequentially pseudocompact. 

Then the product of any family of feebly compact (resp., Tychonoff pseudocompact) spaces all of which satisfy
$\mathcal P$ is feebly compact (resp., pseudocompact).
\end{corollary} 

\begin{proof}
Immediate from Theorem \ref{infprod}, since sequential pseudocompactness 
trivially implies feeble compactness,
and hence,  in 
the (productive) class of Tychonoff spaces,
 also pseudocompactness.
\end{proof}

In \cite{AMPRT} many  properties 
$\mathcal P$ are shown to satisfy the assumption in 
Corollary \ref{p}, among them, the properties of being 
a topological group, of being Tychonoff and scattered, or first countable, 
 or $ \psi$-$\omega$-scattered.

In particular, from Theorem \ref{infprod} and \cite[Proposition 1.10]{AMPRT}, 
we get another proof of the classical result by 
Comfort and Ross that any product
of pseudocompact topological groups is pseudocompact.
 However, it is not clear
whether this is a real simplification: it might be the case that any proof that every pseudocompact topological
group is sequentially pseudocompact 
already contains enough sophistication to be easily converted into 
a direct proof of Comfort and Ross Theorem.

\begin{corollary} \labbel{prodp}
If $X = \prod _{h \in H} X_h$  is a product of 
Tychonoff  pseudocompact 
topological spaces and, for each $h \in H$,
$X_h$ is either scattered, or first countable, or, more generally, $ \phi$-$\omega$-scattered,
or allows the structure of a topological group, then 
$X$ is pseudocompact (actually, sequentially pseudocompact). 
 \end{corollary}

 \begin{proof} 
By the mentioned results from \cite{AMPRT}, each $X_h$
is sequentially pseudocompact, hence $X$ is sequentially pseudocompact by Theorem \ref{infprod}.   
\end{proof}

\begin{problems} \labbel{prob}
 \begin{enumerate}   
 \item 
Find other properties $\mathcal P$ satisfying the assumption in 
Corollary \ref{p}, besides those described in \cite{AMPRT}.
\item
Is there some significant part of the (topological)
theory of pseudocompact topological groups
 which follows already from the assumption of sequential pseudocompactness?
More precisely, are there other theorems holding for
pseudocompact topological groups which can be generalized
to sequentially pseudocompact topological spaces (with no
algebraic structure on them)?
  \end{enumerate}
 \end{problems}

The proof of Theorem \ref{infprod} actually shows a little more. 
If $(X_h) _{h \in H} $ is a family of topological spaces, 
the \emph{$ \omega$-box topology} on   
 $ \prod _{h \in H} X_h$ is defined as the topology a base of which 
consists of the set of the form
$   \prod _{h \in H} O_h$,
where each $O_h$ is an open set of $X_h$, 
and $|\{ h \in H \mid O_h \not = X_h\}| \leq \omega $.

\begin{proposition} \labbel{box}
Suppose that
$ (X_h)_{h \in H}$
  is a family of 
sequentially pseudocompact 
topological spaces.
If
$( O_n) _{n \in \omega } $ 
is a sequence of  nonempty
open sets in the 
$ \omega$-box topology on $ \prod _{h \in H} X_h$,
then 
there are an infinite set $J \subseteq \omega $
and a point $x \in  \prod _{h \in H} X_h$ such that
$\{ n \in J \mid U \cap O_n = \emptyset \}$ is finite,
for every
neighborhood $U$ of $x$ in the Tychonoff product topology 
on $\prod _{h \in H} X_h$.
\end{proposition}

\begin{proof} 
Same as the   proof of Theorem \ref{infprod}.
Indeed, the set
$K$ in the proof of (A) is countable anyway, this time being the countable union 
of a family of countable sets.
Thus the proposition holds if and only if it holds
for every countable $H$. Now we can argue 
as in  the second part of the proof of Theorem \ref{infprod}.
Notice that there  we have only used 
the assumption that, for each $n$, 
$O'_n$ has the form 
$\prod _{i \in \omega } O' _{n,i} $, 
for open sets
$O' _{n,i} $  of $X_i$,
with no need of assuming that (for fixed $n$)
$O' _{n,i} =X_i$, for all but finitely many
$i \in \omega $. 
\end{proof} 

Of course, in the statement of  Proposition \ref{box}, 
the neighborhoods $U$ of $x$ have to be considered in the Tychonoff 
product topology. The statement would turn out to be false 
allowing $U$ vary among the neighborhoods of $x$ in the $ \omega$-box
topology. Indeed, for example,  
a discrete two-element space is vacuously  sequentially pseudocompact;
however, its $ \omega ^\text{th} $ power
in the $ \omega$-box
topology
is a discrete space
of cardinality $ \mathfrak c$, hence the conclusion
of Proposition \ref{box} would fail.

We shall present now a  version of Theorem \ref{infprod}
in which the assumption that all factors are sequentially pseudocompact 
can be relaxed (with a correspondingly weaker conclusion, of course).

\begin{prop} \labbel{prod1} 
  \begin{enumerate}    
\item    
The product of a Tychonoff sequentially pseudocompact
topological space  with a Tychonoff pseudocompact space is pseudocompact.
\item
More generally, the product of a sequentially pseudocompact
topological space  with a feebly compact 
space is feebly compact.
\item
If all factors of some product
are sequentially pseudocompact, except possibly for one factor
which is feebly compact, then the product is
feebly compact.
\end{enumerate} 
\end{prop} 

\begin{proof}
We shall prove (2), which is more general than (1),
by the results recalled in Remark \ref{rmk},
and since the product of two Tychonoff spaces is still Tychonoff.

Suppose that $X$ is sequentially pseudocompact, and $Y$ is feebly compact, 
and let 
$( O_n) _{n \in \omega } $ 
be a sequence of nonempty
open sets of $X \times Y$.
Thus, there are nonempty opens sets $( A_n) _{n \in \omega } $ 
in $X$ and nonempty open sets 
$( B_n) _{n \in \omega } $ in $Y$ such that 
$O_n \supseteq A_n \times B_n$, for every $n \in  \omega$. 

Since $X$ is sequentially pseudocompact, then, by Proposition \ref{equiv}, 
there are an infinite set $J \subseteq \omega $
and a point $x \in X$ such that every
neighborhood of $x$ intersects all but finitely many elements of 
$( A_n) _{n \in J} $.
By applying the feeble compactness of $Y$ to the sequence
$( B_n) _{n \in J} $, we get $y \in Y$ such that every neighborhood 
of $y$ intersects infinitely many elements of 
$( B_n) _{n \in J} $.
The above conditions imply that every neighborhood of $(x,y)$
in $X \times Y$ intersects infinitely many elements
of
$( A_n \times B_n) _{n \in J} $,
hence infinitely many elements from the original sequence
$( O_n) _{n \in \omega } $. Feeble compactness of $X \times Y$
is thus proved.

(3) is immediate from (2) and Theorem \ref{infprod} 
(grouping together the sequentially pseudocompact factors).
 \end{proof}  

Condition (1) in Proposition \ref{prod1}
is also  a consequence of
\cite[Proposition 1.9]{AMPRT}.
The same arguments on
\cite[p. 7]{AMPRT},
together with \cite[Theorem 1.2]{V}
(which uses no separation axiom)
furnish another proof of (2).
Condition (3) is apparently  
new.
The particular case of Proposition \ref{prod1} in which the assumption of
sequential pseudocompactness is strengthened to sequential compactness  
is well known \cite{S}.

The following notions might deserve some study, in particular when $\alpha$
is a cardinal.

\begin{definition} \labbel{lambda} 
For $ \alpha $ an infinite limit ordinal, we say that a topological space $X$ is 
\emph{sequentially $ \alpha $-pseudocompact} if and only if, for any sequence 
$( O_ \beta ) _{ \beta   \in \alpha   } $ 
 of  nonempty
open sets of $X$, there 
are 
some $x \in X$ and
a subset $Z$ of $\alpha$ such that 
$Z$ has  order type $\alpha$, and,
for every neighborhood $U$ of $x$, there is 
$\beta < \alpha $ such that  
 $ U \cap O_ \gamma  \not = \emptyset $,
for every $\gamma \in Z$
such that $\gamma> \beta $.  

If we modify the above definition by further requesting that the
$( O_ \beta )$'s are pairwise disjoint, we say that 
 $X$ is 
\emph{d-\brfrt sequentially $ \alpha $-\brfrt pseudocompact}.
Clearly, for every $\alpha$, sequential $ \alpha $-pseudocompactness
implies  d-\brfrt sequential $ \alpha $-pseudocompactness, and 
for $\alpha= \omega $, 
both notions are equivalent (and equivalent to 
sequential pseudocompactness),
 by Proposition \ref{equiv}.
\end{definition}

\end{document}